\documentclass[]{interact}

\usepackage{epstopdf}
\usepackage[caption=false]{subfig}

\usepackage[numbers,sort&compress]{natbib}
\bibpunct[, ]{[}{]}{,}{n}{,}{,}
\makeatletter
\def\NAT@def@citea{\def\@citea{\NAT@separator}}
\makeatother

\theoremstyle{plain}
\newtheorem{theorem}{Theorem}[section]
\newtheorem{lemma}[theorem]{Lemma}
\newtheorem{corollary}[theorem]{Corollary}
\newtheorem{proposition}[theorem]{Proposition}

\theoremstyle{definition}
\newtheorem{definition}[theorem]{Definition}
\newtheorem{example}[theorem]{Example}

\theoremstyle{remark}

\begin{document}

\articletype{}

\title{Upper semi-Weyl and upper semi-Browder spectra of unbounded upper triangular operator matrices}

\author{
\name{Wurichaihu Bai\textsuperscript{a,b}, Qingmei Bai\textsuperscript{c} and Alatancang Chen\textsuperscript{c$\ast$}\thanks{*Corresponding author. Email: alatanca@imu.edu.cn}}
\affil{\textsuperscript{a}School of Mathematical Sciences, Inner Mongolia University, Hohhot, China; \textsuperscript{b}School of Mathematical Sciences, Inner Mongolia Normal University ,Hohhot,  China;\textsuperscript{c}Hohhot Minzu College, Hohhot, China}
}

\maketitle

\begin{abstract}
In this paper, we study the unbounded upper
triangular operator matrix with diagonal domain. Some sufficient and necessary
conditions are given under which upper semi-Weyl spectrum (resp. upper semi-Browder
spectrum) of such operator matrix is equal to the union of the upper semi-Weyl spectra (resp. the upper
semi-Browder spectra) of its diagonal entries. As an application, the corresponding spectral properties of Hamiltonian operator matrix are obtained.\end{abstract}

\small\bf MR(2010) Subject Classification\ \ {\rm 47A53, 47A11}
\begin{keywords}
unbounded upper triangular operator matrix; upper semi-Weyl spectrum; upper semi-Browder spectrum; Hamiltonian operator matrix
\end{keywords}

\section{Introduction}

 Let $\mathcal{K},\mathcal{H}$ be the infinite dimensional separable Hilbert spaces and $\mathcal{C}(\mathcal{H},\mathcal{K})
 (\mathcal{C}^+(\mathcal{H},\mathcal{K}))$ be the set of all closed(closable) linear operators from $\mathcal{H}$ into $\mathcal{K}$. We also write $\mathcal{C}(\mathcal{H},\mathcal{H})(\mathcal{C}^+(\mathcal{H},\mathcal{H}))$ as $\mathcal{C}(\mathcal{H})(\mathcal{C}^+(\mathcal{H}))$. Let $G$ a linear subspace in $\mathcal{H}$. Then $\overline{G}$ and $G^\bot$ denote the closure and the orthogonal complement of $G$, respectively. For a (linear) operator
$T$ between Hilbert spaces, we use $\mathcal{D}(T), \mathcal{R}(T)$ and $\mathcal{N}(T)$ to denote the domain, the range
and the kernel of $T$, and write $\alpha(T)$ and $\beta(T)$ for the dimensions of the kernel $\mathcal{N}(T)$ and
the quotient space $\mathcal{H}/\mathcal{R}(T)$, respectively. According to [1], an operator $T\in \mathcal{C}(\mathcal{H},\mathcal{K})$ with dense domain is Fredholm, which can be defined as follows. An operator $T\in \mathcal{C}(\mathcal{H},\mathcal{K})$ with dense domain is said to be upper semi-Fredholm (resp., lower semi-Fredholm)
if $\alpha(T)<\infty$(resp.,$\beta(T)<\infty$) and $\mathcal{R}(T)$ is closed. If both $\alpha(T)$ and $\beta(T)$ are finite, then $T$ is called Fredholm operator. We call that $T$ is upper semi-Weyl (resp., lower semi-Weyl) if it is upper semi-Fredholm (resp., lower semi-Fredholm) with the index
$ind(T ) = \alpha(T)-\beta(T)\leq 0$ (resp., $ind(T)\geq 0$) and $T$ is upper semi-Browder (resp., lower semi--Browder) if it is upper semi-Fredholm (resp., lower semi-Fredholm) of finite ascent $asc(T)$(resp.,finite
descent $dsc(T)$), where
\begin{center}
$asc(T ) = inf\{n\mid \mathcal{N}(T^n)= \mathcal{N}(T^{n+1})\}$, \\ $ dsc(T ) = inf\{n\mid \mathcal{R}(T^n)= \mathcal{R}(T^{n+1})\}$with $inf\emptyset= \infty$.
\end{center}
We call that $T$ is Weyl if it is Fredholm with
$ind(T ) = 0$. Then, the upper semi-Weyl spectrum, upper semi-Browder spectrum of $T$ are, respectively,
defined by
\begin{center}
$\sigma_{SF_{+}^{-}}(T)=\{\lambda\in\mathbb{C}:T-\lambda$ is not upper semi-Weyl\},
\end{center}
\begin{center}
$\sigma_{lb}(T)=\{\lambda\in\mathbb{C}:T-\lambda$ is not upper semi-Browder\}.
\end{center}

Block operator matrices play a major role in coupled systems of partial differential
equations, and their spectral properties are of concerned interest.
Especially, the study of upper triangular operator matrices and related subjects is one of
the hottest areas in operator theory. Recently, a number of mathematicians have
studied $2\times2$ bounded upper triangular operator matrices (see, e.g. [2-4]). In [5-9], the
authors, making use of the single-valued extension property, estimated the defect sets
($\sigma_*(A) \cup \sigma_*(D))\backslash\sigma_*(T)$ and obtained some sufficient conditions for

~~~~~~~~~~~~~~~~~~~~~~~~~~~~~~~~~~~~~~$\sigma_*(T) =\sigma_*(A) \cup \sigma_*(D), $~~~~~~~~~~~~~~~~~~~~~~~~~~~~~~~~~~~(1.1)
where
\begin{center}
$T=\left[
      \begin{array}{cc}
        A & B \\
      0 & D \\
      \end{array}
    \right]$
\end{center}
is a bounded operator matrix acting on Banach space and $\sigma_* \in
\{\sigma_e, \sigma_w, \sigma_b\}$. In [10], the
authors extend these results to unbounded case . The main aim of this paper is to get sufficient
and necessary conditions for (1.1) of an unbounded operator with $\sigma_* \in
\{\sigma_{SF_{+}^{-}}, \sigma_{lb}\}$. One of the significant differences between
unbounded and bounded operator matrices arises in their domains. In general, one could not
get certain spectral properties of unbounded operator matrix $T$ using the factorization
\begin{center}
$T=\left[
      \begin{array}{cc}
        I & 0 \\
      0 & D \\
      \end{array}
    \right]\left[
      \begin{array}{cc}
        I & B \\
      0 & I \\
      \end{array}
    \right]\left[
      \begin{array}{cc}
        A & 0 \\
      0 & I \\
      \end{array}
    \right]$
\end{center}
where
\begin{center}
$T=\left[
      \begin{array}{cc}
        A & B \\
      0 & D \\
      \end{array}
    \right]: \mathcal{D}(A)\oplus \mathcal{D}(D)\subset \mathcal{H} \oplus \mathcal{H} \longrightarrow \mathcal{H}\oplus \mathcal{H}$
\end{center}
is a closed operator matrix.

Applying different method -- space decomposition technique -- we present some sufficient
and necessary conditions for (1.1) in this paper. More precisely, the defect sets
($\sigma_*(A) \cup \sigma_*(D))\backslash\sigma_*(T)$ with $\sigma_* \in
\{\sigma_{SF_{+}^{-}}, \sigma_{lb}\}$ are actually described, and, in addition,
these results are applied to a Hamiltonian operator matrix.

\begin {definition} (see [11]) A closed (linear) operator
\begin{center}
$H=\left[
      \begin{array}{cc}
        A & B \\
      C & -A^* \\
      \end{array}
    \right]: \mathcal{D}(H)\subset  \mathcal{H} \oplus \mathcal{H} \longrightarrow \mathcal{H}\oplus \mathcal{H}$
\end{center}
with dense domain
\begin{center}
$\mathcal{D}(H) = \mathcal{D}(A)\cap \mathcal{D}(C)\oplus \mathcal{D}(B)\cap \mathcal{D}(A^*)$,
\end{center}
is called a Hamiltonian operator matrix, if $A$ is a closed operator with dense domain and
$B,C$ are self-adjoint operators.
\end {definition}
\begin {definition} (see [12]) Let $T$ and $B$ be linear operators from $\mathcal{H}$ to $\mathcal{K}$. We say that $B$ is $T$-compact if

(1) $\mathcal{D}(T)\subset\mathcal{D}(B)$ and

(2) $B$ is compact on $\mathcal{H}_T$,\\
where $\mathcal{H}_T$ denotes $\mathcal{D}(T)$ endowed with the graph norm i.e. $\parallel x\parallel_T=\parallel x\parallel+\parallel Tx\parallel$, for $x\in\mathcal{D}(T)$.
\end {definition}
\section{Some properties of upper triangular operator matrix}
\begin{lemma} (see [12])
 Suppose that $T\in \mathcal{C}(\mathcal{H})$ is a Fredholm operator and $B$ is
$T$-compact. Then

(1) $T + B$ is a Fredholm operator,

(2) $ind(T + B) = ind(T )$.
\end{lemma}
\begin {lemma}
Let \begin{center}
$T =\left[
      \begin{array}{cc}
        A & B \\
      0 & D \\
      \end{array}
    \right]: \mathcal{D}(A)\oplus \mathcal{D}(D)\subset \mathcal{H} \oplus \mathcal{K} \longrightarrow \mathcal{H} \oplus \mathcal{K}$
\end{center}
be closed operator matrix such that $A\in \mathcal{C}(\mathcal{H}), D\in \mathcal{C}(\mathcal{K})$ with dense domains and let $B\in \mathcal{C}^+_{D}(\mathcal{K},\mathcal{H})=\{B\in \mathcal{C}^+(\mathcal{K},\mathcal{H}):\mathcal{D}(D)\subset\mathcal{D}(B), D\in C(K)\}$, then there exists some $B\in \mathcal{C}^+_{D}(\mathcal{K},\mathcal{H})$ such that $T$ is upper semi-Weyl operator if and only if $A$ is upper semi-Fredholm operator and
\begin{center}
$\begin {cases}
\alpha(D)<\infty~$and$~\alpha(A)+\alpha(D)\leq \beta(A)+\beta(D) \\
~or~\beta(A)=\alpha(D)=\infty &\mbox{if $\mathcal{R}(D)$ is closed}\\
   \beta(A)=\infty &\mbox{if $\mathcal{R}(D)$ is not closed}
\end{cases}$
\end{center}
\end {lemma}
\begin{proof}
Suppose that $T$ is upper semi-Weyl operator for some $B\in \mathcal{C}^+_{D}(\mathcal{K},\mathcal{H})$. Then
$A$ is upper semi-Fredholm operator.

If $\beta(A)<\infty$, then $A$ is Fredholm operator. Also since $T$ is upper semi-Fredholm operator, we have $D$ is upper semi-Fredholm operator. In fact, $T$ can be written as follows:

\begin{center}
$T =\left[
 \begin{array}{cccc}
        A_1 & 0&B_1 & B_3 \\
      0  & 0&B_2 & B_4 \\
      0  & 0&D_1 & 0 \\
      0  & 0&0 & 0 \\
 \end{array}
\right]: \left(
 \begin{array}{c}
      \mathcal{D}(A)\cap\mathcal{N}(A)^\bot \\
       \mathcal{N}(A)\\
       \mathcal{D}(D)\cap\mathcal{N}(D)^\bot \\
       \mathcal{N}(D)\\
\end{array}
\right)
       \longrightarrow
\left(
\begin{array}{c}
      \mathcal{R}(A)\\
      \mathcal{R}(A)^\bot \\
       \overline{\mathcal{R}(D)}\\
      \mathcal{R}(D)^\bot \\
\end{array}
\right)$,
\end{center}
$A_1: \mathcal{D}(A)\cap\mathcal{N}(A)^\bot \longrightarrow\mathcal{R}(A)$ is a bijection, then there exists operator
\begin{center}
$Q =\left[
 \begin{array}{cccc}
        I & 0&-A^{-1}_1B_1 & -A^{-1}_1B_3 \\
      0  & I&0 & 0 \\
      0  & 0&I & 0 \\
      0  & 0&0 & I \\
 \end{array}
\right]: \left(
 \begin{array}{c}
      \mathcal{R}(A)\\
      \mathcal{R}(A)^\bot \\
       \overline{\mathcal{R}(D)}\\
      \mathcal{R}(D)^\bot \\
\end{array}
\right)
       \longrightarrow
\left(
\begin{array}{c}
      \mathcal{R}(A)\\
      \mathcal{R}(A)^\bot \\
       \overline{\mathcal{R}(D)}\\
      \mathcal{R}(D)^\bot \\
\end{array}
\right)$
\end{center}
such that
\begin{center}
$TQ =\left[
 \begin{array}{cccc}
        A_1 & 0&0 & 0 \\
      0  & 0&B_2 & B_4 \\
      0  & 0&D_1 & 0 \\
      0  & 0&0 & 0 \\
 \end{array}
\right]=\left[
 \begin{array}{cccc}
        A_1 & 0&0 & 0 \\
      0  & 0&0 & 0 \\
      0  & 0&D_1 & 0 \\
      0  & 0&0 & 0 \\
 \end{array}
\right]+\left[
 \begin{array}{cccc}
        0 & 0&0 & 0 \\
      0  & 0&B_2 & B_4 \\
      0  & 0&0 & 0 \\
      0  & 0&0 & 0 \\
 \end{array}
\right]$.
\end{center}

Since $\beta(A)<\infty$, then $B_2, B_4$ are compact operators, therefore $D$ is upper semi-Fredholm operator from Lemma 2.1. Hence $ind(T )=ind(A)+ind(D)\leq0$ i.e. $\alpha(A)+\alpha(D)\leq \beta(A)+\beta(D)$. So when $T$ is upper semi-Weyl operator for some $B\in \mathcal{C}^+_{D}(\mathcal{K},\mathcal{H})$, we have $A$ is upper semi-Fredholm operator and
\begin{center}
$\begin {cases}
\beta(A)=\infty \\
or~D~is~upper~semi-Fredholm~operator~and~\alpha(A)+\alpha(D)\leq \beta(A)+\beta(D).
\end{cases}$
\end{center}
Again
\begin{center}
$\begin {cases}
\beta(A)=\infty \\
or~D~is~upper~semi-Fredholm~operator~and~\alpha(A)+\alpha(D)\leq \beta(A)+\beta(D).
\end{cases}$
\end{center}
if and only if
\begin{center}
$\begin {cases}
\alpha(D)<\infty~$and$~\alpha(A)+\alpha(D)\leq \beta(A)+\beta(D) \\
or~\beta(A)=\alpha(D)=\infty &\mbox{if $\mathcal{R}(D)$ is closed}\\
   \beta(A)=\infty &\mbox{if $\mathcal{R}(D)$ is not closed}
\end{cases}$
\end{center}

Conversely, if $A, D$ are upper~semi-Fredholm operators and $\alpha(A)+\alpha(D)\leq \beta(A)+\beta(D)$, then $T$ is upper semi-Weyl operator for every $B\in \mathcal{C}^+_{D}(\mathcal{K},\mathcal{H})$, from [10, Lemma 2.2].

If $A$ is upper~semi-Fredholm operator and $\beta(A)=\infty$, then there exist two infinite dimensional subspaces  $M, N$ of $\mathcal{R}(A)^\bot$ such that $\mathcal{R}(A)^\bot=M \oplus N$. We define an operator
\begin{center}
$B =\left[
 \begin{array}{c}
        0  \\
      U   \\
      0 \\
 \end{array}
\right]:
\mathcal{K}
       \longrightarrow
\left(
\begin{array}{c}
      \mathcal{R}(A)\\
      M \\
       N\\
\end{array}
\right)$,
\end{center}
where $U:\mathcal{K}
       \longrightarrow M$ is an unitary operator. Therefore $T$ is upper semi-Weyl operator and $\alpha(A)=\alpha(T)$. In fact, if $\left(
\begin{array}{c}
      x\\
      y \\
\end{array}
\right)\in \mathcal{N}(T)$, then $Ax+By=0$ and $Dy=0$. Since $Ax=-By\in\mathcal{R}(A)\cap M=\{0\}$, then $Ax=0$ and $By=0$. And we obtain $y=0$ by the definition of $U$. Therefore $\mathcal{N}(T)\subset\mathcal{N}(A)\oplus\{0\}$. Moreover $\alpha(T)=\alpha(A)<\infty$. The following proof is that $\mathcal{R}(A)$ is a closed set. Since $\mathcal{R}(A)$ is a closed set, then we have to prove that $\mathcal{R}\left(
\begin{array}{c}
      U\\
      0 \\
      D \\
\end{array}
\right)$ is a closed set. Let $\left(
\begin{array}{c}
      U\\
      0 \\
      D \\
\end{array}
\right)x_n\longrightarrow \left(
\begin{array}{c}
      y_1\\
      y_2 \\
      y_3\\
\end{array}
\right)$, then $U x_n\longrightarrow y_1, D x_n\longrightarrow y_3$ and $y_2=0$. By the definition of $U$ and the closeness of $U$, we have $U^{-1}y_1\in\mathcal{D}(D)$ and $y_3=DU^{-1}y_1$, i.e. $\left(
\begin{array}{c}
      U\\
      0 \\
      D \\
\end{array}
\right) U^{-1}y_1=\left(
\begin{array}{c}
      y_1\\
      y_2 \\
      y_3\\
\end{array}
\right)\in\mathcal{R}\left(
\begin{array}{c}
      U\\
      0 \\
      D \\
\end{array}
\right)$, so $\mathcal{R}\left(
\begin{array}{c}
      U\\
      0 \\
      D \\
\end{array}
\right)$ is a closed set. Since $dim N=\infty$, then $\beta(T)=\infty$, therefore $T$ is upper semi-Weyl operator.
\end{proof}
\begin{lemma} Let $T =\left[
      \begin{array}{cc}
        A & B \\
      0 & D \\
      \end{array}
    \right]: \mathcal{D}(A)\oplus \mathcal{D}(D)\subset \mathcal{H} \oplus \mathcal{K} \longrightarrow \mathcal{H} \oplus \mathcal{K}$ be closed operator matrix such that $A\in \mathcal{C}(\mathcal{H}), D\in \mathcal{C}(\mathcal{K})$ with dense domains and let $B\in \mathcal{C}^+_{D}(\mathcal{K},\mathcal{H})$, then\\
$(1)$ if $A$ and $D$ are upper semi-Weyl, then $T$ is upper semi-Weyl.\\
$(2)$ if $A$ and $D$ are upper semi-Browder, then $T$ is upper semi-Browder.
\end{lemma}
\begin{proof}
It is easily obtained by Lemma 2.2 of [10] and [13].
\end{proof}

\section{Main results}

\begin{theorem} Let $T =\left[
      \begin{array}{cc}
        A & B \\
      0 & D \\
      \end{array}
    \right]: \mathcal{D}(A)\oplus \mathcal{D}(D)\subset \mathcal{H} \oplus \mathcal{K} \longrightarrow \mathcal{H} \oplus \mathcal{K}$ be closed operator matrix such that $A\in \mathcal{C}(\mathcal{H}), D\in \mathcal{C}(\mathcal{K})$ with dense domains and let $B\in \mathcal{C}^+_{D}(\mathcal{K},\mathcal{H})$, then \begin{center}
    $\sigma_{SF_{+}^{-}}(A)\cup\sigma_{SF_{+}^{-}}(D)=\sigma_{SF_{+}^{-}}(T)\cup(\sigma_{p_{+}}(D)\cap\sigma_{p_{+}}(A^*)^-)\cup(\sigma_{p_{+}}(A)\cap\sigma_{p_{+}}(D^*)^-)
    \cup(\sigma_{p_{\infty}}(A^*)^-\cap\sigma_{p_{\infty}}(D)),$
    \end{center}
where $\sigma_{p_{+}}(\cdot)=\{\lambda\in\sigma_p(\cdot):\alpha(\cdot-\lambda I)>\beta(\cdot-\lambda I)\}, \sigma_{p_{\infty}}(\cdot)=\{\lambda\in\sigma_p(\cdot):\alpha(\cdot-\lambda I)=\infty\}$ and $\sigma_{p_{+}}(\cdot)^-=\{\overline{\lambda}\in \mathbb{C}:\lambda\in\sigma_{p_{+}}(\cdot)\}$.
\end{theorem}
\begin{proof}
Let $\lambda\in(\sigma_{SF_{+}^{-}}(A)\cup\sigma_{SF_{+}^{-}}(D))\backslash\sigma_{SF_{+}^{-}}(T)$, then $T-\lambda I$ is upper semi-Fredholm and $ind(T-\lambda I)\leq 0$.

(1) If $ind(T-\lambda I)=0$, then $\lambda\in(\sigma_{p_{+}}(D)\cap\sigma_{p_{+}}(A^*)^-)\cup(\sigma_{p_{+}}(A)\cap\sigma_{p_{+}}(D^*)^-)
    $ from Theorem 3.2 of [10].

(2) If $ind(T-\lambda I)<0$, then $A-\lambda I$ is upper semi-Fredholm operator and
\begin{center}
$\begin {cases}
\alpha(A-\lambda I)+\alpha(D-\lambda I)\leq \beta(A-\lambda I)+\beta(D-\lambda I) ~$and$~\\
\alpha(D-\lambda I)<\infty ~or~\beta(A-\lambda I)=\alpha(D-\lambda I)=\infty &\mbox{if $\mathcal{R}(D-\lambda I)$ is closed}\\
   \beta(A-\lambda I)=\infty &\mbox{if $\mathcal{R}(D-\lambda I)$ is not closed}
\end{cases}$
\end{center}
by Lemma 2.2.

If $\lambda\in\sigma_{SF_{+}^{-}}(A)$, then $ind(A-\lambda I)>0$, so $\lambda\in\sigma_{p_{+}}(A)$. From Lemma 2.2, if $\mathcal{R}(D-\lambda I)$ is closed, we have $D-\lambda I$ is upper semi-Fredholm and $ind(D-\lambda I)<0$ or $\beta(A-\lambda I)=\alpha(D-\lambda I)=\infty$ (but it is impossible), so $\lambda\in\sigma_{p_{+}}(D^*)^-$. If $\mathcal{R}(D-\lambda I)$ is not closed, then $\beta(A-\lambda I)=\infty$, it is impossible. Therefore $\lambda\in\sigma_{p_{+}}(A)\cap\sigma_{p_{+}}(D^*)^-$.

If $\lambda\in\sigma_{SF_{+}^{-}}(D)$, then $D-\lambda I$ is not upper semi-Fredholm or $ind(D-\lambda I)>0$.

Let $D-\lambda I$ is not upper semi-Fredholm, then $\mathcal{R}(D-\lambda I)$ is not closed or $\alpha(D-\lambda I)=\infty$, moreover, $\beta(A-\lambda I)=\infty$ by Lemma 2.2. Thus $\lambda\in\sigma_{p_{\infty}}(A^*)^-\cap\sigma_{p_{\infty}}(D)$.

Let $D-\lambda I$ is upper semi-Fredholm and $ind(D-\lambda I)>0$, then $ind(A-\lambda I)<0$, so $\lambda\in\sigma_{p_{+}}(D)\cap\sigma_{p_{+}}(A^*)^-$.

Conversely, from Lemma 2.3, we have $\sigma_{SF_{+}^{-}}(A)\cup\sigma_{SF_{+}^{-}}(D)\supseteq\sigma_{SF_{+}^{-}}(T)$ and $(\sigma_{p_{+}}(D)\cap\sigma_{p_{+}}(A^*)^-)\cup(\sigma_{p_{+}}(A)\cap\sigma_{p_{+}}(D^*)^-)
    \cup(\sigma_{p_{\infty}}(A^*)^-\cap\sigma_{p_{\infty}}(D))\subseteq\sigma_{SF_{+}^{-}}(A)\cup\sigma_{SF_{+}^{-}}(D)$.

The proof is complete.
\end{proof}
\begin{corollary} Let $T =\left[
      \begin{array}{cc}
        A & B \\
      0 & D \\
      \end{array}
    \right]: \mathcal{D}(A)\oplus \mathcal{D}(D)\subset \mathcal{H} \oplus \mathcal{K} \longrightarrow \mathcal{H} \oplus \mathcal{K}$ be closed operator matrix such that $A\in \mathcal{C}(\mathcal{H}), D\in \mathcal{C}(\mathcal{K})$ with dense domains and let $B\in \mathcal{C}^+_{D}(\mathcal{K},\mathcal{H})$, then \begin{center}
    $\sigma_{SF_{+}^{-}}(A)\cup\sigma_{SF_{+}^{-}}(D)=\sigma_{SF_{+}^{-}}(T)$
    \end{center}
    if and only if
    \begin{center}
    $\sigma_{p_{+}}(D)\cap\sigma_{p_{+}}(A^*)^-  \subseteq\sigma_{SF_{+}^{-}}(T)$ and $\sigma_{p_{+}}(A)\cap\sigma_{p_{+}}(D^*)^-\subseteq\sigma_{SF_{+}^{-}}(T)$,
 $\sigma_{p_{\infty}}(A^*)^-\cap\sigma_{p_{\infty}}(D)\subseteq\sigma_{SF_{+}^{-}}(T)$.
    \end{center}
 In particular, if $\sigma_{p_{+}}(D)\cap\sigma_{p_{+}}(A^*)^- =\emptyset$ and $\sigma_{p_{+}}(A)\cap\sigma_{p_{+}}(D^*)^-=\emptyset$,
 $\sigma_{p_{\infty}}(A^*)^-\cap\sigma_{p_{\infty}}(D)=\emptyset$, then $\sigma_{SF_{+}^{-}}(A)\cup\sigma_{SF_{+}^{-}}(D)=\sigma_{SF_{+}^{-}}(T)$.
\end{corollary}

\begin{theorem} Let $T =\left[
      \begin{array}{cc}
        A & B \\
      0 & D \\
      \end{array}
    \right]: \mathcal{D}(A)\oplus \mathcal{D}(D)\subset \mathcal{H} \oplus \mathcal{K} \longrightarrow \mathcal{H} \oplus \mathcal{K}$ be closed operator matrix such that $A\in \mathcal{C}(\mathcal{H}), D\in \mathcal{C}(\mathcal{K})$ with dense domains and let $B\in \mathcal{C}^+_{D}(\mathcal{K},\mathcal{H})$, then \begin{center}
    $\sigma_{lb}(A)\cup\sigma_{lb}(D)=\sigma_{lb}(T)\cup (\sigma_{p_{\infty}}(A^*)^-\cap\sigma_{p_{\infty}}(D))\cup\sigma_{asc}(D)$,
    \end{center}
where $\sigma_{asc}(\cdot)=\{\lambda\in\mathbb{C}:asc(\cdot-\lambda I)=\infty\}$.
\end{theorem}
\begin{proof}
Let $\lambda\in(\sigma_{lb}(A)\cup\sigma_{lb}(D))\backslash\sigma_{lb}(T)$, then $T-\lambda I$ is upper semi-Fredholm and $asc(T-\lambda I)<\infty$, therefore $A-\lambda I$ is upper semi-Fredholm and $asc(A-\lambda I)<\infty$, from [13]. i.e. $\lambda$ is not belong to $\sigma_{lb}(A)$, so $\lambda\in\sigma_{lb}(D)$,i.e. $D-\lambda I$ is not upper semi-Fredholm or $asc(D-\lambda I)=\infty$. If $\mathcal{R}(D-\lambda I)$ is not closed or $\alpha(D-\lambda I)=\infty$, then $\beta(A-\lambda I)=\infty$ by Lemma 2.2. Therefore $\lambda\in\sigma_{p_{\infty}}(A^*)^-\cap\sigma_{p_{\infty}}(D)$. If $D-\lambda I$ is upper semi-Fredholm but $asc(D-\lambda I)=\infty$, then $\lambda\in\sigma_{asc}(D)$. So $(\sigma_{lb}(A)\cup\sigma_{lb}(D))\backslash\sigma_{lb}(T)\subset \sigma_{p_{\infty}}(A^*)^-\cup\sigma_{asc}(D)$.

Conversely, we have $\sigma_{lb}(A)\cup\sigma_{lb}(D)\supseteq\sigma_{lb}(T)$ by Lemma 2.3,  and $(\sigma_{p_{\infty}}(A^*)^-\cap\sigma_{p_{\infty}}(D))\cup\sigma_{asc}(D)\subseteq\sigma_{lb}(A)\cup\sigma_{lb}(D)$ is easily obtained.
\end{proof}
\begin{corollary} Let $T =\left[
      \begin{array}{cc}
        A & B \\
      0 & D \\
      \end{array}
    \right]: \mathcal{D}(A)\oplus \mathcal{D}(D)\subset \mathcal{H} \oplus \mathcal{K} \longrightarrow \mathcal{H} \oplus \mathcal{K}$ be closed operator matrix such that $A\in \mathcal{C}(\mathcal{H}), D\in \mathcal{C}(\mathcal{K})$ with dense domains and let $B\in \mathcal{C}^+_{D}(\mathcal{K},\mathcal{H})$, then \begin{center}
    $\sigma_{lb}(A)\cup\sigma_{lb}(D)=\sigma_{lb}(T)$
    \end{center}
    if and only if
    \begin{center}
    $\sigma_{asc}(D)\subseteq\sigma_{lb}(T)$ and
 $\sigma_{p_{\infty}}(A^*)^-\cap\sigma_{p_{\infty}}(D)\subseteq\sigma_{lb}(T)$.
    \end{center}
 In particular, if  $\sigma_{asc}(D)=\emptyset$ and
 $\sigma_{p_{\infty}}(A^*)^-\cap\sigma_{p_{\infty}}(D) =\emptyset$, then $\sigma_{lb}(A)\cup\sigma_{lb}(D)=\sigma_{lb}(T)$.
\end{corollary}

\section{applications}
In this section, we obtained some properties of Hamiltonian operator matrix.
\begin{proposition} Let
\begin{center}
$H=\left[
      \begin{array}{cc}
        A & B \\
      0 & -A^* \\
      \end{array}
    \right]: \mathcal{D}(A)\oplus \mathcal{D}(A^*)\subset  \mathcal{H} \oplus \mathcal{H} \longrightarrow \mathcal{H}\oplus \mathcal{H}$
\end{center}
 be a Hamiltonian operator matrix. Then
 \begin{center}
    $\sigma_{SF_{+}^{-}}(A)\cup\sigma_{SF_{+}^{-}}(-A^*)=\sigma_{SF_{+}^{-}}(H)$
    \end{center}
    if and only if
    \begin{center}
    $\sigma_{p_{+}}(-A^*)\cap\sigma_{p_{+}}(A^*)^-  \subseteq\sigma_{SF_{+}^{-}}(H)$ and $\sigma_{p_{+}}(A)\cap\sigma_{p_{+}}((-A^*)^*)^-\subseteq\sigma_{SF_{+}^{-}}(H)$,
 $\sigma_{p_{\infty}}(A^*)^-\cap\sigma_{p_{\infty}}(D)\subseteq\sigma_{SF_{+}^{-}}(H)$.
    \end{center}
 In particular, if $\sigma_{p_{+}}(-A^*)\cap\sigma_{p_{+}}(A^*)^- =\emptyset$ and $\sigma_{p_{+}}(A)\cap\sigma_{p_{+}}((-A^*)^*)^-=\emptyset$,
 $\sigma_{p_{\infty}}(A^*)^-\cap\sigma_{p_{\infty}}(D)=\emptyset$, then $\sigma_{SF_{+}^{-}}(A)\cup\sigma_{SF_{+}^{-}}(-A^*)=\sigma_{SF_{+}^{-}}(H)$.
\end{proposition}
\begin{proposition} Let
\begin{center}
$H=\left[
      \begin{array}{cc}
        A & B \\
      0 & -A^* \\
      \end{array}
    \right]: \mathcal{D}(A)\oplus \mathcal{D}(A^*)\subset  \mathcal{H} \oplus \mathcal{H} \longrightarrow \mathcal{H}\oplus \mathcal{H}$
\end{center}
 be a Hamiltonian operator matrix. Then
 \begin{center}
    $\sigma_{lb}(A)\cup\sigma_{lb}(-A^*)=\sigma_{lb}(H)$
    \end{center}
    if and only if
    \begin{center}
    $\sigma_{asc}(-A^*)\subseteq\sigma_{lb}(H)$ and
 $\sigma_{p_{\infty}}(A^*)^-\cap\sigma_{p_{\infty}}(D)\subseteq\sigma_{lb}(H)$.
    \end{center}
 In particular, if  $\sigma_{asc}(-A^*)=\emptyset$ and
 $\sigma_{p_{\infty}}(A^*)^-\cap\sigma_{p_{\infty}}(D)=\emptyset$, then $\sigma_{lb}(A)\cup\sigma_{lb}(-A^*)=\sigma_{lb}(H)$.
\end{proposition}

\begin{example} Consider the plate bending equation in domain $\{(x, y) : 0 < x < 1, 0 <
y < 1\}$
\begin{center}
$D(\frac{\partial^2}{\partial x^2}+\frac{\partial^2}{\partial y^2})^2\omega=0,$
\end{center}
with boundary conditions
\begin{center}
$\omega(x, 0) = \omega(x, 1) = 0,$ \\
$\frac{\partial^2\omega}{\partial x^2}+\frac{\partial^2\omega}{\partial y^2}= 0, y = 0, 1.$
\end{center}
Set
\begin{center}
$\theta =\frac{\partial\omega}{\partial x }, q = D(\frac{\partial^3\omega}{\partial x^3}+\frac{\partial^3\omega}{\partial y^3}), m =-D(\frac{\partial^2\omega}{\partial x^2}+\frac{\partial^2\omega}{\partial y^2}),$
\end{center}
then the equation can be written as the following Hamiltonian system [14]
\begin{center}
$\frac{\partial}{\partial x }\left[
 \begin{array}{c}
        \omega \\
      \theta \\
      q\\
      m \\
 \end{array}
\right]=\left[
 \begin{array}{cccc}
        0 & 1&0 & 0 \\
      - \frac{\partial^2}{\partial y^2} & 0&0 & -\frac{1}{D} \\
      0  & 0&0 & \frac{\partial^2}{\partial y^2} \\
      0  & 0&-1 & 0 \\
 \end{array}
\right]\left[
 \begin{array}{c}
        \omega \\
      \theta \\
      q\\
      m \\
 \end{array}
\right]$
\end{center}
and the corresponding Hamiltonian operator matrix is given by
\begin{center}
$H=\left[
 \begin{array}{cccc}
        0 & 1&0 & 0 \\
      - \frac{d^2}{d y^2} & 0&0 & -\frac{1}{D} \\
      0  & 0&0 & \frac{d^2}{d y^2} \\
      0  & 0&-1 & 0 \\
 \end{array}
\right]:\left[
      \begin{array}{cc}
        A & B \\
      0 & -A^* \\
      \end{array}
    \right]$
\end{center}
with domain is $\mathcal{D}(A)\oplus \mathcal{D}(A^*)\subset  \mathcal{H} \oplus \mathcal{H}$, where $\mathcal{H}= L^2(0, 1) \oplus L^2(0, 1),\mathcal{A} = AC[0, 1],$
and
\begin{center}
$A =\left[
      \begin{array}{cc}
       0 & 1 \\
      - \frac{d^2}{d y^2} & 0\\
      \end{array}
    \right], B =\left[
      \begin{array}{cc}
       0 & 0 \\
  0 & -\frac{1}{D}\\
      \end{array}
    \right],$\\
$\mathcal{D}(A)=\{\left[
      \begin{array}{c}
       \omega \\
 \theta \\
      \end{array}
    \right]\in\mathcal{H}: \omega(0) = \omega(1) = 0, \omega^{'}\in\mathcal{A}, \omega^{''}\in\mathcal{H}\}.$
\end{center}
Through a simple calculation, we know $\sigma_{p_{\infty}}(A^*)=\emptyset, \sigma_{p_{+}}(-A^*)\cap\sigma_{p_{+}}(A^*)^-=\emptyset,
-\sigma_{p_{+}}(A)^-\cap\sigma_{p_{+}}(A)=\emptyset$ and $\sigma_{asc}(A^*)=\emptyset$. Then from Propositions 4.1 and 4.2, we have
\begin{center}
$-\sigma_{\ast}(A^*)\cup\sigma_{\ast}(A)=\sigma_{\ast}(H)$,
\end{center}
where $\sigma_{\ast}\in\{\sigma_{SF_{+}^{-}}, \sigma_{lb}\}$.
\end{example}
\section*{Disclosure statement}

No potential conflict of interest was reported by the authors.

\section*{Funding}

This work is supported by National Natural Science Foundation of China
[grant number 11561053],[grant number 11761029]; Natural Science Foundation of Inner Mongolia[grant number 2018BS01001]; Research Program of Sciences at Universities of Inner Mongolia Autonomous Region[grant number NJZZ18018],[grant number NJZY18021]; Subject of Research Foundation of Inner Mongolia Normal University of China [grant number 2016ZRYB001].

\end{document}